\documentclass[10pt,a4paper]{amsart}
\usepackage{amsfonts,amsmath,amssymb}
\usepackage{graphics}
\usepackage[dvipdfmx]{graphicx}

\allowdisplaybreaks[2]

\newtheorem{theorem}{Theorem}

\newtheorem{proposition}[theorem]{Proposition}
\newtheorem{corollary}[theorem]{Corollary}

\theoremstyle{definition}
\newtheorem{definition}[theorem]{Definition}

\newcommand{\bbE}{\mathbb E}

\newcommand{\bbZ}{\mathbb Z}
\newcommand{\bbR}{\mathbb R}
\newcommand{\bbH}{\mathbb H}

\newcommand{\sgn}{\mathrm{sgn}}

\numberwithin{equation}{section}

\begin{document}

\title[Whitney-type Formula for Non-null-homotopic Curves]{Whitney-type Formula for Non-null-homotopic Curves on Aspherical Surfaces}

\author{Masayuki YAMASAKI}

\address{Department of Applied Science, Okayama University of Science, 
Okayama, Okayama 700-0005, Japan}
\email{yamasaki@surgery.matrix.jp}
\urladdr{http://surgery.matrix.jp/}

\begin{abstract}
In an earlier paper \cite{MY}, I defined a new winding number of regular closed
curves on complete euclidean/hyperbolic surfaces and showed that this winding number,
together with the free homotopy class, determines the regular homotopy class.
In this paper, I give a Whitney-type formula for the winding number of
non-null-homotopic generic regular closed curves on surfaces with a complete
euclidean or hyperbolic structure,
generalizing the formula for curves on a torus by Tanio and Kobayashi
\cite{T-K}.
\end{abstract}

\maketitle


\section{Introduction}
In \cite{MY}, I introduced a certain winding number of regular closed curves
on aspherical surfaces with a good geometric structure, and showed that
two homotopic regular closed curves are regularly homotopic if and only if they 
have the same winding number.
In this paper, we assume that the surface has a complete euclidean or hyperbolic structure,
and give a Whitney-type formula for this winding number
when the curve is non-null-homotopic and {\it generic} in the sense that
the self-intersection points are all transversal double points.
Such a formula was given by Tanio and Kobayashi in \cite{T-K} for curves on a torus
with a flat riemannian metric.
They defined a regular homotopy invariant $t(\gamma)$ of a generic regular closed curve
$\gamma$ on an orientable surface $M$ as follows:
Let $d$ be a double point of $\gamma$.
Then $d$ divides $\gamma$ into two loops based at $d$.  $\gamma_d$ denotes the
loop which comes back to $d$ from the left, and $\gamma'_d$ denotes the other loop 
which comes back to $d$ from the right.  Then $t(\gamma)$ is defined by
\[
t(\gamma)=\#\{d| [\gamma_d]=0\in H_1(M;\bbZ)\}
-\#\{d| [\gamma'_d]=0\in H_1(M;\bbZ)\}.
\]
They showed that, in the case of a torus, this invariant $t(\gamma)$ is equal to the
winding number of $\gamma$ if $\gamma$ is not null-homotopic.
See \cite{T} by Tanio for other properties of $t(\gamma)$.
We define a regular homotopy invariant $I(\gamma)$
by replacing the homology group $H_1(M;\bbZ)$ by $\pi_1(M,*)$ in the definition above,
when $M$ is orientable.
So $t(\gamma)$ and $I(\gamma)$ coincide if the fundamental group of the surface is abelian.
In \S2, I assume that $\gamma$ is a non-null-homotopic generic regular closed curve
on an orientable surface with complete euclidean/hyperbolic structure,
and give a description of $I(\gamma)$ which can be easily adapted to 
the non-orientable surface case, and show that it is equal to the winding number
$W(\gamma)$ defined in \cite{MY}.
In \S3, I will extend this result to the case when $M$ is non-orientable.

Please note that the term `winding number' is used in various meanings in the literature.
J.~Roe's textbook \cite{JR}, for example, defines the `winding number' of a planar closed curve
to be the number of times the curve winds around a given point,
and the term `rotation number' is used for the number of times the direction map of
a regular closed curve winds around the origin 0.
This `rotation number' is also called
the `Whitney index', the `tangent winding number', and, unfortunately, the `winding number'.
In this article, a `winding number' always mean either the `Whitney index' or its generalization.
Please also note that there are various generalizations of the Whitney index; 
for example, see the articles \cite{BLR} by Reinhart and \cite{DRJC} by Chillingworth.
My generalization is different from theirs, and is closer to 
the generalization by Kobayashi \cite{OK}.

To express a closed curve on a surafce $M$, I used a parametric representation
$\gamma:[a,b]\to M$ in \cite{MY}, and will be mainly using a map $\gamma:S^1\to M$ from
the unit circle in this paper, for technical reasons.
Please forgive my abuse of notation.
We assume that the base point of $\gamma:S^1\to M$ is $\gamma((1,0))$.

\bigbreak
\section{Preparations and the Orientable Surface Case}
We first look at the case when the curve is null-homotopic.
If a curve $\gamma$ is null-homotopic, then it lifts to a regular closed curve
$\widetilde\gamma$ on the universal cover $\widetilde M$ of the surface $M$.
We denote the universal covering map of $M$ by $p_M$.
A regular homotopy of $\gamma$ induces a regular homotopy of $\widetilde\gamma$,
and vice versa.  So the regular homotopy classification of null-homotopic
regular closed curves is the same as that of regular closed curves on the universal cover.
If $M$ is a complete euclidean or hyperbolic surface, then 
$\widetilde M$ is equal to the euclidean plane $\bbE^2$ or the hyperbolic plane $\bbH^2$
(more precisely, the whole plane or the upper half plane/the open unit disk).
A regular homotopy in $\bbE^2$ of curves in $\widetilde M$ can be deformed
into a regular homotopy in $\widetilde M$ without changing the curves;
so, the classical Whitney index \cite{HW} for the euclidean plane can be used
for the classification on $\widetilde M$.
If $M$ is orientable, then all the lifts of $\gamma$ have the same Whitney index;
on the other hand, if $M$ is non-orientable, then either all of them have the trivial
Whitney index or the lifts are divided into two types of lifts whose Whitney indices have 
the same absolute value and the opposite signs.
So, given a null-homotopic regular closed curve $\gamma$ in $M$,
we defined its winding number $W(\gamma)$ in \cite{MY} by:
\[
W(\gamma)=\begin{cases} W(\widetilde\gamma)\in\bbZ  & \text{if $M$ is orientable,}\\
|W(\widetilde\gamma)|\in \bbZ_+ & \text{if $M$ is non-orientable,}
\end{cases}
\]
where $\widetilde \gamma$ is any lift of $\gamma$
and $W(\widetilde\gamma)$ is its classical Whitney index of $\widetilde\gamma$.
The classical Whitney's formula \cite{HW} expresses $W(\widetilde\gamma)$ in terms of the
signs of the double points, with respect to a suitably chosen base point.

Next let us assume that $\gamma$ is non-null-homotopic and generic.
In this case, $\gamma$ does not lift to the universal cover; but,
the composite map $\gamma\circ p_{S^1}:\bbR=\widetilde{S^1} \to M$ does lift to 
a map $\widetilde \gamma:\bbR\to \widetilde M$, where $p_{S^1}:\bbR\to S^1$
is the universal covering map of $S^1$ given by $p_{S^1}(t)=(\cos 2\pi t, \sin 2\pi t)$.
Such a lift $\widetilde\gamma:\widetilde{S^1}\to\widetilde M$ will be called a {\it cover} of
$\gamma$.  Since $\gamma$ is generic, any cover $\widetilde\gamma$ is also generic.
Also note that any cover can be obtained by composing a given cover and an appropriate
deck transformation of $\widetilde M$.

Let $D(\gamma)$ be the set of all the double points of $\gamma$.
Let $d\in D(\gamma)$ be any element.
At $d$, the curve $\gamma$ splits into two closed loops $\gamma_1$ and $\gamma_2$
based at $d$.
Since $\gamma$ is non-null-homotopic, at least one of $\gamma_1$ and $\gamma_2$
must be non-null-homotopic.
The followings are obviously equivalent:
\begin{itemize}
\item Either $\gamma_1$ or $\gamma_2$ is null-homotopic.
\item There exists a cover $\widetilde\gamma$ of $\gamma$ and a double point $\widetilde{d}$ of 
$\widetilde\gamma$ such that $p_M(\widetilde{d})=d$.
\item For any cover $\widetilde\gamma$ of $\gamma$, there exists a double point $\widetilde{d}$
of $\widetilde\gamma$ such that $p_M(\widetilde{d})=d$.
\end{itemize}
We define $D_\pm(\gamma)$ to be the subset of $D(\gamma)$ consisting of those double points
satisfying the conditions above, and define $D_0(\gamma)$ to be its complement.
In other words, double points in $D_\pm(\gamma)$ correspond to self-intersections
of the covers, and the double points in $D_0(\gamma)$ correspond to 
mutual intersections of covers with distinct images.

For each point $d\in D(\gamma)$, we wish to define 
$\sgn(\gamma,d)\in\{-1,0,+1\}$ 
so that the winding number of $\gamma$ is equal to the sum 
$I(\gamma)=\sum_{d\in D(\gamma)} \sgn(\gamma,d)$, and we will succeed in some cases.
For a double point  $d\in D_0(\gamma)$, we define $\sgn(\gamma,d)$ to be 0.
So, the sum $I(\gamma)$ above will actually be equal to
$\sum_{d\in D_\pm(\gamma)} \sgn(\gamma,d)$.

We first define $\sgn(\widetilde\gamma,\widetilde d)$
of a double point $\widetilde d$ of a cover $\widetilde\gamma$ of $\gamma$.
Let $t_1<t_2$ be the real numbers such that 
$\widetilde\gamma(t_1)=\widetilde\gamma(t_2)=\widetilde d$. 
Choose a small positive number $\delta$.
We define $\sgn(\widetilde\gamma, \widetilde d)$ to be $+1$ (resp. $-1$)
if the arc $\widetilde\gamma((t_2-\delta, t_2+\delta))$ 
(drawn horizontally in Fig.\ref{fig:sign})
crosses the arc $\widetilde\gamma((t_1-\delta, t_1+\delta))$ 
(drawn vertically in Fig.\ref{fig:sign})
from left to right (resp. from right to left).
\begin{figure}[htbp]
  \begin{center}
   \includegraphics{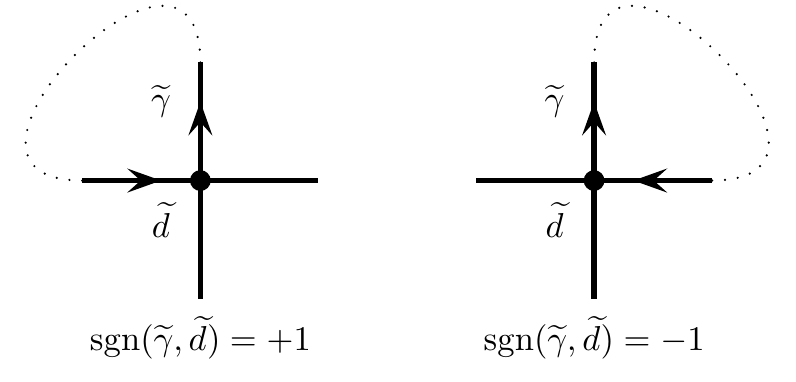}
  \end{center}
\caption{Sign convention for a double point of $\widetilde\gamma$.}\label{fig:sign}
\end{figure}
Let $G_M=\pi_1(M,*)$ be the group of the deck transformations of $p_M$,
and let $w: G_M\to \{\pm1\}$ be the orientation homomorphism.
For $T\in G_M$, the following identity holds:
\[
\sgn(T\circ \widetilde\gamma, T(\widetilde d))=w(T)
\sgn(\widetilde\gamma, \widetilde d).
\]
Let $H_{\widetilde\gamma}$ denote the infinite cyclic subgroup of $G_M$ generated by 
the deck transformation $T_0$ which sends $\widetilde\gamma(0)$ to $\widetilde\gamma(1)$, 
then $H_{\widetilde\gamma}$ acts freely on the image $\widetilde\gamma(\bbR)$;
therefore, the following identity holds:
\[
\sgn(\widetilde\gamma, T_0(\widetilde d))=w([\gamma])
\sgn(\widetilde\gamma, \widetilde d).
\]

If we deform $\gamma$ by a regular homotopy, then it lifts to regular homotopies of the covers,
and if a birth/death of two double poits $d$, $e$ of $\gamma$ occurs, then corresponding 
births/deaths occur for self- and/or  mutual intersections of the covers, and
each birth/death pair in $\widetilde M$ must be one of the following:
\begin{itemize}
\item[(1)] mutual intersections of covers with distinct images,
\item[(2)] double points of the same cover with opposite signs.
\end{itemize}
In case (1), $d$ and $e$ are both points in $D_0(\gamma)$, and, in case (2),
they are both points in $D_\pm(\gamma)$.

\medbreak
In the rest of this section, we consider the case when $M$ is orientable.
The non-orientable case will be treated in the next section.

Let us define $\sgn(\gamma,d)$ for $d\in D_\pm(\gamma)$.
First fix a cover $\widetilde \gamma$ of $\gamma$;
then, take  any double point $\widetilde d$ of $\widetilde\gamma$ such that
$p_M(\widetilde d)=d$.
Now set $\sgn(\gamma,d)=\sgn(\widetilde\gamma,\widetilde d)$.
This does not depend on the choice of $\widetilde\gamma$ or $\widetilde d$.

The following gives a Whitney-type formula for the winding number 
$W(\gamma)\in\bbZ$ geometrically defined in \cite{MY}
in the case when $M$ is orientable.

\begin{theorem}   \label{thm:orientable} Suppose that
$M$ is a complete euclidean/hyperbolic orientable surface
and that $\gamma:S^1\to M$ is a non-null-homotopic generic regular closed curve.
Then
\[
W(\gamma)=\sum_{d\in D(\gamma)} \sgn(\gamma,d)=\sum_{d\in D_\pm(\gamma)} \sgn(\gamma,d)
\]
\end{theorem}

\begin{proof} 
We first show that $I(\gamma)=\sum_{d\in D_\pm(\gamma)} \sgn(\gamma,d)$ is
invariant under regular homotopies of $\gamma$.
Fix a cover $\widetilde\gamma$.
For each $d\in D_\pm(\gamma)$, choose any double point 
$\widetilde d$ of $\widetilde\gamma$ such that $p_M(\widetilde d)=d$.
Then the set $D(\widetilde\gamma)$ of the double points of $\widetilde\gamma$
is the disjoint union
\[
\bigsqcup_{T\in H_{\widetilde\gamma}}\{ T(\widetilde d)~|~d\in D_\pm(\gamma)\}.
\]
Also note that, for $T\in H_{\widetilde\gamma}$, we have
$\sgn(\widetilde\gamma, T(\widetilde d))=\sgn(\widetilde\gamma, \widetilde d)$, 
because $T$ preserves orientation.
A regular homotopy of a curve on a surface may change the sum $I(\gamma)$
only when there are births/deaths of double points in $D_\pm(\gamma)$.
But corresponding $G_M$-equivariant births/deaths pairs for $\widetilde\gamma$
must have opposite signs.  Thus $I(\gamma)$ is invariant under regular homotopies.

Next we show that $W(\gamma)=I(\gamma)$.  Let $k$ denote the integer $W(\gamma)$.
Pick a non-null-homotopic regular closed curve $\gamma$ on $M$.
Let $L_\gamma$ be the set of the length of closed curves freely homotopic to $\gamma$.
There are two cases.  

If $\inf L_\gamma>0$, then there is a closed geodesic $\gamma_0$ on $M$ 
which is homotopic to $\gamma$, and its winding number $W(\gamma_0)$ is 0 \cite{MY}.
Slightly perturb $\gamma_0$ by a regular homotopy to obtain a generic regular closed curve
$\gamma_1$.  Then add $k$ small kinks to $\gamma_1$ to obtain a 
generic regular closed curve $\gamma_2$.
Then $\gamma_2$ is homotopic to $\gamma$ and its winding number is equal to $k$.
Therefore, $\gamma$ and $\gamma_2$ are regularly homotopic \cite{MY}.
Covering these steps in $\widetilde M$, we obtain 
a cover $\widetilde \gamma_2$ of $\gamma_2$ as follows: first take
a cover $\widetilde\gamma_0$ of the geodesic $\gamma_0$;
it has no double points, because it is a geodesic on $\bbE^2$ or $\bbH^2$.
Then slightly perturb it by an $H_{\widetilde\gamma_0}$-equivariant regular homotopy
to obtain $\widetilde\gamma_1$.  It still has no double points, so, $I(\gamma_1)$ is 0.
Then $H_{\widetilde\gamma_0}$-equivariantly add appropriate kinks 
to obtain $\widetilde\gamma_2$.
Then we see that $I(\gamma_2)=k$.
By the regular homotopy invariance of $I(\text{--})$, we obtain
$I(\gamma)=k=W(\gamma)$.

Next let us suppose that $\inf L_\gamma$ is 0.  In this case, 
$M$ is hyperbolic and $\gamma$ is homotopic to a holocycle $\gamma_0$ around a cusp.
The winding number of $\gamma_0$ is 0 \cite{MY}.  The rest of the argument is similar to
the $\inf L_\gamma >0$ case, and this completes the proof.
\end{proof}

\begin{corollary}   \label{cor:orientable} Let $M$ be as above, and assume $\gamma$ and  $\gamma'$ are
non-null-homotopic generic regular closed curves on $M$.  Then 
$\gamma$ and $\gamma'$ are regularly homotopic if and only if the following holds:
\begin{itemize}
\item[(1)] $\gamma$ and $\gamma'$ are freely homotopic, and
\item[(2)] $I(\gamma)=I(\gamma')$.
\end{itemize}
\end{corollary}

\bigbreak
\section{The Non-orientable Surface Case}

In this section, we assume that $M$ is non-orientable.

\subsection{The case when $\gamma$ is orientation reversing}
When a regular closed curve on $M$ is orientation reversing, the generalized winding 
number $W(\gamma)$ was defined to be an element of $\bbZ/2\bbZ$.

For $d\in D(\gamma)$, we define $\sgn(\gamma,d)\in\{-1,0,+1\}$ as in the 
previous section.  When $d\in D_\pm(\gamma)$, its sign depend on the choice
of $\widetilde\gamma$ and $\widetilde d$, but its mod 2 value is well-defined.
The argument in the previous section can be used to prove the following.

\begin{theorem} \label{thm:ori-rev} Suppose that
$M$ is a complete euclidean/hyperbolic non-orientable surface
and that $\gamma:S^1\to M$ is an orientation reversing generic regular closed curve.
Then
\[
W(\gamma)=\sum_{d\in D(\gamma)} \sgn(\gamma,d) +2\bbZ
=\sum_{d\in D_\pm(\gamma)} \sgn(\gamma,d) +2\bbZ
\in \bbZ/2\bbZ
\]
\end{theorem}

\begin{corollary}  \label{cor:ori-rev} 
Let $M$ be as above, and assume $\gamma$ and  $\gamma'$ are
orientation reversing generic regular closed curves on $M$.  Then 
$\gamma$ and $\gamma'$ are regularly homotopic if and only if the following holds:
\begin{itemize}
\item[(1)] $\gamma$ and $\gamma'$ are freely homotopic, and
\item[(2)] the numbers of elements of $D_\pm(\gamma)$ and 
$D_\pm(\gamma')$ have the same parity.
\end{itemize}
\end{corollary}

\bigbreak
\subsection{The case when $\gamma$ is orientation preserving} 
In this case, we cannot define $\sgn(\gamma, d)$ which is independent of the
choice of the cover $\widetilde\gamma$.

Let $\widetilde\gamma$ be a fixed cover of $\gamma$.  
Take a double point $d\in D_\pm(\gamma)$ and a point $\widetilde d$ on 
$\widetilde\gamma$ satisfying $p_M(\widetilde d)=d$.
Since $\gamma$ is orientation preserving, the value of $\sgn(\widetilde\gamma,\widetilde d)$
is independent of the choice of $\widetilde d$.
The sum
\[
i_{\widetilde\gamma}(\gamma)=\sum_{d\in D(\gamma)}\sgn(\widetilde\gamma,\widetilde d)\in\bbZ.
\]
is independent of the choice of $\widetilde d$'s, but it does depend on the choice
of the cover $\widetilde\gamma$.
If $\widehat \gamma$ is another cover of $\gamma$, then there is a deck transformation
$T$ such that $\widehat\gamma=T\circ\widetilde\gamma$, and we have
\[
\sgn(\widehat\gamma, \widehat d)=w(T) \sgn(\widetilde\gamma,\widetilde d),
\]
where $\widehat d=T(\widetilde d)$; therefore, we have
\[
i_{\widehat\gamma}(\gamma)=w(T) i_{\widetilde\gamma}(\gamma).
\]
On the other hand, 
if we deform $\gamma$ by a regular homotopy, then $\widetilde\gamma$ is deformed
by the lift of the regular homotopy which starts from $\widetilde\gamma$;
and this does not change the value of $i_{\widetilde\gamma}(\gamma)$.

There are two cases: the case when $\gamma$ is reversible and the case when $\gamma$ is not
reversible.  Let us recall the notion of reversibility from \cite{MY}.
\begin{definition} \label{def:rev}
(1) Let $G$ be a group with a fixed non-trivial homomorphism $w:G\to\{\pm1\}$.
An element $g\in G$ is said to be {\it reversible} if there is an element $h$ of the 
centralizer $C_G(g)$ of $g$ in $G$ such that $w(h)=-1$.
\\
(2) Let $M$ be a non-orientable surface.  A loop $\gamma$ on $M$ based at $p\in M$ is said to be
{\it reversible} if the element $[\gamma]\in\pi_1(M,p)$ is reversible
with respect to the orientation homomorphism $w:\pi_1(M,p)\to\{\pm1\}$.
If $\xi$ is an orientation reversing closed curve on $M$ based at $p$ such that
$[\xi^*\gamma\xi]=[\gamma]\in \pi_1(M,p)$ ({\it i.e.} $[\xi]\in
C_{\pi_1(M,p)}([\gamma])$), then we say that $\xi$ {\it reverses} $\gamma$.
Here $\xi^*$ denotes the inverse of the path $\xi$.
\end{definition} 

{\bf Let us first assume that $\gamma$ is reversible}.  In this case, we difine
\[
I(\gamma)=|i_{\widetilde\gamma}(\gamma)|.
\]
This does not depend on the choice of the cover $\widetilde\gamma$, and it is easy to check that
this is a regular homotopy invariant.  
So the following can be checked as in the previous section.

\begin{theorem} \label{thm:reversible} Suppose that
$M$ is a complete euclidean/hyperbolic non-orientable surface
and that $\gamma:S^1\to M$ is an orientation preserving non-null-homotopic
generic regular closed curve.
Further assume that $\gamma$ is reversible.
Then
\[
W(\gamma)=I(\gamma) \in \bbZ
\]
\end{theorem}

\begin{corollary}  \label{cor:reversible} 
Let $M$ be as above, and assume $\gamma$ and  $\gamma'$ are
orientation preserving non-null-homotopic generic regular closed curves on $M$
and are reversible.
Then 
$\gamma$ and $\gamma'$ are regularly homotopic if and only if the following holds:
\begin{itemize}
\item[(1)] $\gamma$ and $\gamma'$ are freely homotopic, and
\item[(2)] $I(\gamma)=I(\gamma')$
\end{itemize}
\end{corollary}

\bigbreak
{\bf Finally let us assume that $\gamma$ is not reversible}.
In this case, we could not define $W(\gamma)\in\bbZ$ in \cite{MY},
because we had to choose the sign.
More precisely, we needed to fix an orientation preserving non-null-homotopic 
non-reversible closed loop $\gamma_0$ of $M$
and needed to fix a lift 
$\widetilde\gamma_0:[a,b]\to\widetilde M$ to define
the invariant $W_{\widetilde\gamma_0}(\gamma)$
for a regular closed curve $\gamma$ which is freely homotopic to $\gamma_0$.

The same is true for ``$I(\gamma)$''.
In our situation, given $\gamma$, we may choose $\gamma_0$ to be either the shortest 
closed geodesic or any holocyle which is freely homotopic to $\gamma$.
Fixing a lift in the paragraph above corresponds to
taking a cover $\widetilde\gamma$ and we can only define
$I_{\widetilde\gamma_0}(\gamma)$ which depends on the choice of
such a $\widetilde\gamma_0$.

By assumption, there is a homotopy from the chosen $\widetilde\gamma_0$ to
a cover $\widetilde\gamma$ of $\gamma$.  Now define
\[
I_{\widetilde\gamma_0}(\gamma)=i_{\widetilde\gamma}(\gamma).
\]
\begin{proposition} \label{prop:wd} $I_{\widetilde\gamma_0}(\gamma)$
does not depend on the free homotopy between $\gamma$ and $\gamma_0$.
\end{proposition}

\begin{proof}
Suppose one homotopy gives a cover $\widetilde\gamma$ and another homotopy
gives another cover $\widehat\gamma$.
Composing the trace of the base point by the first homotopy from $\gamma$ to $\gamma_0$
and the trace of the base point by the second homotopy from $\gamma_0$ to $\gamma$,
we obain a closed curve $\xi$ based at the base point of $\gamma$.
If $w([\xi])=-1$,  then $\xi$ reverses $\gamma$, but this contradicts
the assumption that $\gamma$ is not reversible.  So $w([\xi])=+1$.
This implies that $i_{\widehat\gamma}(\gamma)=i_{\widetilde\gamma}(\gamma)$.
\end{proof}

Now it is easy to show that $I_{\widetilde\gamma_0}(\gamma)$ is a regular homotopy
invariant, and from this the following follows as in the previous cases.

\begin{theorem} \label{thm:non-reversible} Suppose that
$M$ is a complete euclidean/hyperbolic non-orientable surface
and that $\gamma_0:S^1\to M$ is an orientation preserving non-null-homotopic
non-reversible closed curve on $M$.
Fix a cover $\widetilde\gamma_0$ of $\gamma_0$.
If $\gamma:S^1\to M$ is a generic regular closed curve homotopic to $\gamma_0$, 
then we have
\[
W_{\widetilde\gamma_0|[0,1]}(\gamma)=I_{\widetilde\gamma_0}(\gamma) \in \bbZ,
\]
where $\widetilde\gamma_0|[0,1]:[0,1]\to\widetilde M$ denotes 
the restriction of $\widetilde\gamma_0: \bbR\to\widetilde M$
to the interval $[0,1]$.
\end{theorem}

\begin{corollary}  \label{cor:non-reversible} 
Let $M$, $\gamma_0$, and $\widetilde\gamma_0$ be as in the theorem above
and assume that $\gamma$ and  $\gamma'$ are
generic regular closed curves on $M$ which are homotopic to $\gamma_0$.
Then 
$\gamma$ and $\gamma'$ are regularly homotopic if and only if 
$I_{\widetilde\gamma_0}(\gamma)=I_{\widetilde\gamma_0}(\gamma')$
\end{corollary}

\bigbreak


\begin{thebibliography}{99}
%
%

\bibitem{DRJC} D.~R.~J.~Chillingworth,
\textrm{Winding numbers on surfaces, I.},
\textit{Math.~Ann.} \textbf{196}
(1972) 218--249.

\bibitem{OK} O.~Kobayashi,
\textrm{The conformal rotation number},
\textit{Kodai Math.~J.} \textbf{38}
(2015), no.1, 166--171.

\bibitem{BLR} B.~L.~Reinhart,
\textrm{The winding numbers on two manifolds},
\textit{Annales de l'institut Fourier} \textbf{10}
(1960) 271--283.

\bibitem{JR} J.~Roe,
\textit{Winding Around.  
The winding number in topology, geometry, and analysis.}
\textrm{Student Mathematical Library}, \textbf{76}. 
\textrm{American Mathematical Society, Providence, RI; 
Mathematics Advanced Study Semesters, University Park, PA}, 2015.

\bibitem{T-K} H.~Tanio and O.~Kobayashi,
\textrm{Rotation numbers for curves on a torus},
\textit{Geom. Dedicata} \textbf{61}
(1996), no.1, 1--9.

\bibitem{T} H.~Tanio,
\textrm{Regular homotopy invariants of closed curves and the Gauss problem. (Japanese)}
\textrm{Geometry of harmonic maps and submanifolds (Japanese) (Kyoto, 1997).}
\textit{S\=urikaisekikenky\=usho K\=oky\=uroku} No. 995 (1997),
80--94.

\bibitem{HW} H.~Whitney,
\textrm{On regular closed curves in the plane},
\textit{Compositio Math.} \textbf{4}
(1937) 276--284.

\bibitem{MY} M.~Yamasaki,
\textrm{Winding numbers of regular closed curves on aspherical surfaces},
arXiv:1602.02464v2.

\end{thebibliography}
\end{document}